\documentclass[letterpaper, 10pt, conference]{ieeeconf}

\IEEEoverridecommandlockouts

\usepackage{amsmath}
\usepackage{amssymb}
\usepackage{amsfonts}
\usepackage[export]{adjustbox}

\usepackage{amsthm}
\usepackage{array}
\usepackage[american]{babel}

\usepackage[locale=US]{siunitx}
\DeclareSIUnit\pu{p.u.}

\usepackage{float}
\usepackage[utf8]{inputenc}
\usepackage{mathrsfs}
\usepackage[showonlyrefs=true]{mathtools}
\usepackage{tikz}
\usepackage{xcolor}
\usepackage[caption=false,font=footnotesize]{subfig}

\usepackage{graphicx}
\graphicspath{{./Figs/}}
\DeclareGraphicsExtensions{.pdf}

\DeclareMathAlphabet{\anothermymathbb}{U}{bbold}{m}{n}
\DeclareMathAlphabet{\mymathbb}{U}{BOONDOX-ds}{m}{n}
\DeclareMathOperator{\Imag}{im}
\newcommand{\laplac}{{\bf B}}
\newcommand{\bbR}{\mymathbb{R}}
\newcommand{\bbI}{\mymathbb{I}}
\newcommand{\bbone}{\mymathbb{1}}
\newcommand{\bbzero}{\mymathbb{0}}

\newtheorem{ass}{Assumption}
\newtheorem{lem}{Lemma}
\newtheorem{thm}{Theorem}
\newtheorem*{thm*}{Theorem}
\newtheorem*{crl}{Corollary}

\newtheorem{remark}{Remark}

\title{\LARGE \bf Stability of Dynamic Feedback Optimization\\with Applications to Power Systems}
\author{Sandeep Menta$^{1}$, Adrian Hauswirth$^{1,2}$, Saverio Bolognani$^{1}$, Gabriela Hug$^{2}$, and Florian D\"orfler$^{1}$%
\thanks{This work was supported by ETH Zurich funds and the SNF AP Energy Grant \#160573.}%
\thanks{$^{1}$ Automatic Control Laboratory, ETH Zurich, 8092 Zurich, Switzerland.}%
\thanks{$^{2}$ Power Systems Laboratory, ETH Zurich, 8092 Zurich, Switzerland.}%
\thanks{Corresponding author: Adrian Hauswirth, \textit{hadrian@control.ee.ethz.ch}}%
}

\begin{document}
	
	\maketitle
	\thispagestyle{empty}
    \pagestyle{empty}

	\begin{abstract}	
    We consider the problem of optimizing the steady state of a dynamical system in closed loop.
    Conventionally, the design of feedback optimization control laws assumes that the system is stationary. However, in reality, the dynamics of the (slow) iterative optimization routines can interfere with the (fast) system dynamics. 
    We provide a study of the stability and convergence of these feedback optimization setups in closed loop with the underlying plant, via a custom-tailored singular perturbation analysis result.
    Our study is particularly geared towards applications in power systems and the question whether recently developed online optimization schemes can be deployed without jeopardizing dynamic system stability. 
	\end{abstract}
	
	\section{Introduction}
	
	In this paper, we investigate the problem of designing a real-time feedback controller for a given plant that ensures convergence of the closed-loop trajectories to the solution of a prescribed optimization problem. 
    This approach can be considered a simplified instance of optimal control, in which only a terminal cost is present, and the stage cost (i.e., the cost associated to the transient behavior) is zero. 
    Conversely, the same method can be interpreted as the online implementation of an optimization algorithm, in which intermediate steps are continually applied to the system, rather than being executed offline, purely based on a model of the system.
    
    The main advantages of this approach come from its feedback nature.
    Via the measurement loop, the controlled system responds promptly to unmeasured disturbances and can therefore track the time-varying minimizer of the optimization problem that is provided as a control specification.
    Moreover, these methods are generally more robust against model mismatch and disturbances compared to an offline optimization approach, as the real system is embedded in the feedback loop instead of being simulated.
    Finally, when properly designed, a feedback approach to optimization can ensure that the trajectory of the closed-loop system satisfies input and state constraints and remains feasible while it converges to the desired optimal equilibrium.
    
    This concept has been historically adopted in different fields, for example in process control (see the review in \cite{bonvin2013}) or to engineer congestion control protocols for data networks (starting from the seminal paper \cite{kelly1998}). 
    More recently, this approach has received significant attention in the context of real-time control and optimization of power systems.
    Exploratory works in this direction have shown its potential to provide fast grid operation in order to accommodate fluctuating renewable energy sources, and to yield increased efficiency, safety, and robustness of the system trajectories, even in the presence of a substantial model mismatch.
    For example, feedback optimization has been proposed to drive the state of the power grid to the solution of an AC optimal power flow (OPF) program \cite{dallanese2016} while enforcing operational limits (either as soft constraints \cite{gan2016} or hard ones \cite{hauswirth_projected_2016,hauswirth_online_2017}) and even in the case of time-varying problem parameters \cite{dallanese2018,liu2018}.
    By formulating the problem of frequency regulation and economic re-dispatch as an optimization problem, the feedback optimization approach has also been used to analyze and design better secondary and tertiary frequency control loops \cite{low2014,cady2015,mallada2017}.
    Finally, voltage regulation problems (both to enforce voltage specifications and to maintain a safe distance from voltage collapse) can be formulated via specific cost functions, and therefore can be solved via online feedback optimization schemes \cite{bolognani2015,todescato2018}.
    For a more extensive review of the recent efforts on this topic, we refer to \cite{survey2017}.
    
    The vast majority of these works make a strict assumption on \emph{time-scale separation}, i.e. they consider that the plant is asymptotically stable and sufficiently fast, and thus can be modeled as an algebraic relation.
    As a consequence, the only dynamics observed in the closed-loop system are those induced by the feedback optimization law.
    
    In this paper we eliminate this assumption and explicitly model the plant as a linear time-invariant (LTI) system instead of an algebraic constraint.
    This reveals the detrimental interactions that can occur between the dynamics of the feedback scheme and the natural dynamics of the plant unless control parameters are properly chosen.
    Hence, we propose a design method based on singular perturbation analysis, that guarantees convergence to the minimizers of the given optimization problem while also certifying internal stability of the system.
    
    Our main result takes the form of an easily computable bound on the global controller gain that can be immediately incorporated in the design of the control loop.
    In contrast, related works like \cite{nelson_integral_2017,colombino_online_2018} arrive at complex LMI conditions as stability certificates involving all problem parameters and various degrees of freedom (scaling matrices).
    While these certificates can be evaluated to test stability of the resulting interconnection, they cannot be easily used in the control design phase.
    
    Furthermore, by performing a custom-tailored stability analysis based on a LaSalle argument, we can maintain a high level of generality and require minimal assumptions: The cost function in the given optimization program needs to have compact sublevel sets (no convexity is required), and the LTI system needs to be asymptotically stable (not controllable, nor observable, nor does it need a full rank steady-state map). This is in contrast to \cite{li2018,han_computational_2018}, for example, where the original LTI system needs to realize a gradient or saddle-point algorithm for a convex optimization problem.
    
    The remainder of the work is structured as follows:
    Section~\ref{sec:problemstatement} introduces the problem and the proposed feedback optimization scheme;
    Section~\ref{sec:stab} contains the main results about the stability of the interconnected system and its convergence to the solutions of the given optimization problem; Section~\ref{sec:application} illustrates how those results apply to our application of choice, i.e., frequency control, area balancing, and economic re-dispatch in an electric power grid.
	
	\section{Problem statement}
	\label{sec:problemstatement}
	
	\subsection{Notation}
	
	By $\bbI_n$ and	$\bbzero_{n\times m}$ we denote the identity matrix of size~$n\times n$ and the zero matrix of size $n\times m$, respectively.
	The Euclidean 2-norm of a vector $z \in \bbR^n$ is denoted by $\|z\|$. Given a matrix $A \in \bbR^{n \times m}$,  $\| A \|$ denotes the induced norm of $A$, i.e., $\| A \| := {\sup}_{\| v \| = 1} \| A v \|$. 
	The kernel and image of $A$ are denoted by $\ker A$ and $\Imag A$, respectively.
	If $A$ is symmetric, $A \succ 0$ ($A \succeq 0$) indicates that $A$ is positive definite (semidefinite).
	Given a differentiable function  $f:\bbR^n \rightarrow \bbR$, the gradient $\nabla_{x'} f(x)$ is defined as the column vector of partial derivatives of $f$ at $x$ with respect to the variables~$x'$.

	\subsection{Physical System Description}
	
	Consider an LTI system with dynamics given by
	\begin{subequations} \label{eq:sys_dyn}
	\begin{align}
	    \dot{x} &= Ax + Bu + Q w \\
	    y &= C x + D u \, ,
	\end{align}	
	\end{subequations}
	where $x \in \bbR^n$ denotes the system state, $u \in \bbR^p$ the exogenous input, $w \in \bbR^q$ a constant exogenous disturbance, $y \in \bbR^m$ the output, and the matrices $A, B, Q, C$ and $D$ are of appropriate size.
	
	\begin{ass}
		\label{ass:stab}
		The linear time-invariant system \eqref{eq:sys_dyn} is exponentially asymptotically stable. Namely, there exists a positive definite matrix $P \in \bbR^{n \times n}$ such that
	\begin{equation}\label{eq:ass_lyap}
	    A^TP + PA \preceq - \bbI_n \,.
	\end{equation}
	\end{ass}

	Under Assumption~\ref{ass:stab}, we can conclude that $A$ is invertible and the steady-state input-to-state map for fixed $u$ and $w$ is
	\begin{equation}\label{eq:steady-state_map}
        x = H u + R w,
    \end{equation}
    where $H := - A^{-1} B \in \bbR^{n\times p}$ and $R:= - A^{-1} Q \in \bbR^{n\times q}$.

    \subsection{Steady-state Optimization Problem}
	
	We consider the problem of driving the system output and the control input to a solution of the parametrized optimization problem
	\begin{subequations}\label{eq:base_opt_prob_out}
    \begin{align}
	    \underset{y, u}{\text{minimize}} & \quad \widehat{\Phi}(y, u) \\
	    \text{subject to } & \quad y = C(Hu + Rw) + D u \, ,
    \end{align}
	\end{subequations}
    where $\widehat{\Phi}: \bbR^m \times \bbR^p \rightarrow \bbR$ is a cost function and the disturbance $w$ is assumed to be a fixed parameter. 
    Clearly, the set of solutions to \eqref{eq:base_opt_prob_out} depends on the value of $w$, which is considered unknown.
    
    For the subsequent analysis, it is however beneficial to formulate~\eqref{eq:base_opt_prob_out} without loss of generality in terms of the state~$x$ rather than the output~$y$. Namely, we consider
    \begin{subequations}\label{eq:base_opt_prob}
    \begin{align}
	    \underset{x, u}{\text{minimize}} & \quad \Phi(x, u) \\
	    \text{subject to } & \quad x = Hu + Rw \, , \label{eq:opt_cstr}
    \end{align}
	\end{subequations}
    where $\Phi(x,u):= \widehat{\Phi}(C x + D u, u)$.
    
    Even further, by completely eliminating the constraint \eqref{eq:opt_cstr}, we obtain the optimization problem 
    \begin{equation}\label{eq:base_opt_prob_u}
        \underset{u}{\text{minimize}} \quad \tilde{\Phi}^w(u) \, ,
    \end{equation}
    where $\tilde{\Phi}^w(u) := \Phi(Hu + Rw, u)$. By the chain rule, we have
    \begin{align*}
        \nabla \tilde{\Phi}^w(u) = \widetilde{H}^T \nabla \Phi(Hu + Rw, u) \, ,
    \end{align*}
    where $\widetilde{H}^T := \begin{bmatrix} H^T  & \bbI_p \end{bmatrix}$. In the remainder, we let the cost function $\tilde{\Phi}^w$ satisfy the following very mild assumption.

	\begin{ass}\label{ass:cvx}
		The cost function $\tilde{\Phi}^w(u)$ is differentiable in $u$ for all $u \in \bbR^p$ and all $w \in \bbR^q$, and there exists $\ell > 0$ such that for every $x,x' \in \bbR^n$ and $u \in \bbR^p$ it holds that
		\begin{equation}\label{eq:ass_lip}
		 \left\| \widetilde{H}^T \left( \nabla \Phi (x, u) - \nabla \Phi (x', u) \right) \right\| \leq \ell \| x - x' \| \, .
		\end{equation}
		Moreover, the sublevel sets $\{u \,|\, \tilde{\Phi}^w(u) \le c\}$ of the function $\tilde{\Phi}^w(u)$ are compact for all $c \in \bbR$.
	\end{ass}

	\begin{remark}
	    Assumption~\ref{ass:cvx} is very weak. Differentiability and compactness of level sets are standard requirements for gradient flows to be well-behaved. Condition~\eqref{eq:ass_lip} is satisfied if $\tilde{\Phi}^w$ has an $\ell$-Lipschitz gradient. If only the Lipschitz constant $L$ of $\Phi$ is known, then $\ell := L \| H \|$ will satisfy~\eqref{eq:ass_lip}. However, in cases where $\Phi$ and $H$ are fully known, a tighter bound can often be recovered.
	\end{remark}

	\subsection{Gradient descent flow}
	
	In order to design a controller that steers the system~\eqref{eq:sys_dyn} to the solution of the problem~\eqref{eq:base_opt_prob}, we consider the gradient descent flow on the unconstrained problem \eqref{eq:base_opt_prob_u}, resulting in 
	\begin{align}\label{eq:grad_desc}
        \dot u 
        =  - \epsilon \nabla \tilde{\Phi}^w(u)
        = - \epsilon \widetilde{H}^T \nabla \Phi(Hu + Rw, u) \, ,
    \end{align}
    where $\epsilon > 0$ is a tuning gain to adjust the convergence rate.
    
    Equivalently, this gradient descent flow can be obtained by interconnecting the feedback law
	\begin{equation}\label{eq:fb_law}
	    \dot{u} = -\epsilon {\widetilde{H}}^T \nabla \Phi(x,u) 
	\end{equation}
	with the steady-state map~\eqref{eq:steady-state_map}. 

	Under Assumption~\ref{ass:cvx}, standard results guarantee convergence of~\eqref{eq:grad_desc} (or, equivalently, of the feedback interconnection of \eqref{eq:steady-state_map} and \eqref{eq:fb_law}) to a \emph{critical point}, i.e., to a point that satisfies first-order optimality conditions~\cite{luenberger_linear_1984} (in this case, $\nabla \tilde{\Phi}^w = 0$).
	The condition $\nabla \tilde{\Phi}^w(u^\star) = 0$ implies that
    \begin{equation*}
        \nabla \Phi(x^\star, u^\star) \in \ker \widetilde{H}^T \, ,
    \end{equation*}
    where $x^\star = H u^\star + Rw$.
    This, in turn, corresponds to the first-order optimality conditions of~\eqref{eq:base_opt_prob}, namely
    \begin{equation} \label{eq:firstorderoptimality}
        \nabla \Phi(x^\star, u^\star) \in \Imag \begin{bmatrix} \bbI_n \\ -H^T\end{bmatrix}.
    \end{equation}
    In fact, it is immediate to verify that $\widetilde{H}^T \left[\begin{smallmatrix} \bbI_n \\ -H^T\end{smallmatrix}\right] = 0$. Conversely, if $\nabla \Phi(x^\star, u^\star) \in \ker \widetilde{H}^T$ then 
    $\nabla_u \Phi(x^\star, u^\star) = -H^T \nabla_x \Phi(x^\star, u^\star)$, and therefore \eqref{eq:firstorderoptimality} holds.
    
    This convergence analysis corresponds to studying the stability of the interconnection of the feedback law \eqref{eq:fb_law} with the LTI system \eqref{eq:sys_dyn} (as illustrated in Figure~\ref{inter_sys}) under a \emph{time-scale separation} assumption.
    In the subsequent section, we will analyze the same interconnection without this assumption, in order to derive a prescription on how to tune the gain $\epsilon$ as to guarantee overall system stability. 
    
	\begin{figure}
	    \vspace{.2cm}
	    \centering
	    \includegraphics[width=0.85\columnwidth]{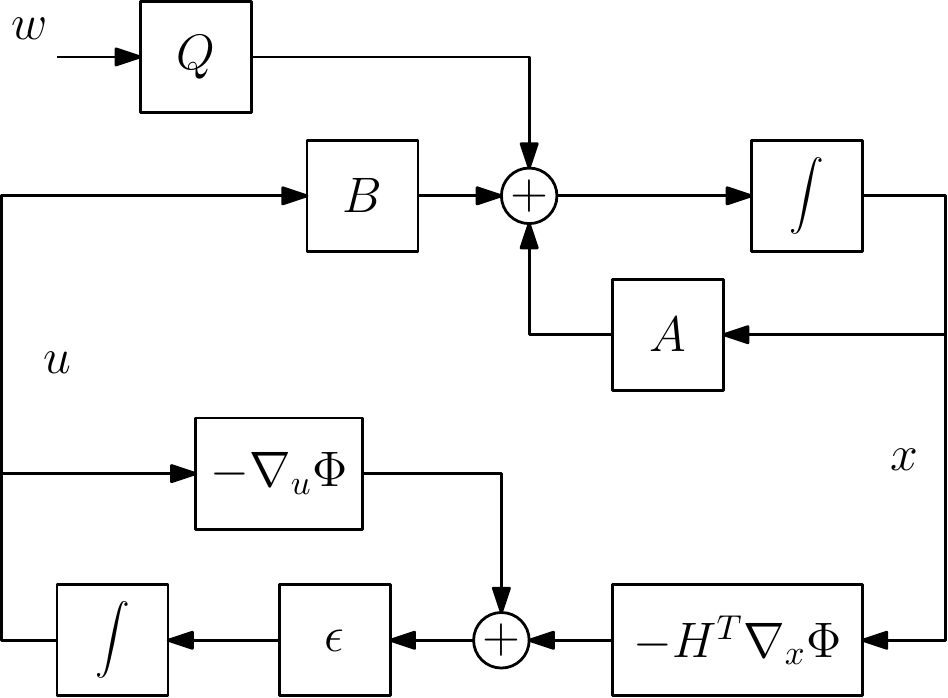}
    	\caption{Block diagram of the interconnected system \eqref{eq:sys_dyn},~\eqref{eq:fb_law}.}
    	\label{inter_sys}
	\end{figure}

	\section{Stability analysis of the gradient descent feedback control}
	\label{sec:stab}

    We now investigate the interconnected system 
    \begin{equation}
    \begin{aligned}\label{eq:ic_sys}
        \dot{x} &= f(x,u,w) &:= Ax + Bu + Q w \\
        \dot{u} & = g(x,u) &:= -\epsilon {\widetilde{H}}^T \nabla \Phi(x,u) 
    \end{aligned}
    \end{equation}
     
    Our main result on the stability of~\eqref{eq:ic_sys} reads as follows.
     
	\begin{thm}\label{thm:main}
		Under Assumptions~\ref{ass:stab} and~\ref{ass:cvx}, the closed-loop system~\eqref{eq:ic_sys} converges to the set of critical points of the optimization problem~\eqref{eq:base_opt_prob} whenever
		\begin{equation}\label{eq:main_bound}
		    \epsilon < \epsilon^\star := \frac{1}{2\ell \beta},
	    \end{equation}
	    where $\ell$ satisfies~\eqref{eq:ass_lip} and $\beta := \| P H \|$.
	    Furthermore, only strict local optimizers of~\eqref{eq:base_opt_prob} can be asymptotically stable. 
	\end{thm}
	
	The bound~\eqref{eq:main_bound} illustrates clearly the trade-off between the dynamic response of the physical system (as measured by $\beta$) and the steepness of the cost function $\Phi$ (as given by $\ell$), and requires that the convergence rate of the gradient flow is smaller than that of the linear system.
	
	Notice also that the convergence to critical points does not guarantee \emph{per se} convergence to a local minimizer, as first-order optimality conditions are necessary but not sufficient for optimality. 
	However, by connecting strict optimality with asymptotic stability, the theorem provides a sufficient condition to identity minimizers, similar to second-order optimality conditions in standard nonlinear programming.
	
	Under stronger conditions on the cost function $\Phi(x,u)$, an immediate conclusion can be drawn.
	
	\begin{crl}
	    Let the assumptions of Theorem~\ref{thm:main} hold.
	    Moreover, let $\Phi(x,u)$ be convex.
	    Then, under the same condition on $\epsilon$, the closed-loop system~\eqref{eq:ic_sys} converges to the set of global solutions of the optimization problem~\eqref{eq:base_opt_prob}.
	\end{crl}

	We split the proof of Theorem~\ref{thm:main} into three parts. First, we propose a LaSalle function that is non-increasing along the trajectories of the system as long as~\eqref{eq:main_bound} is verified. Second, we apply LaSalle's invariance principle and to conclude that all trajectories converge to the set of points that satisfy the first-order optimality conditions of~\eqref{eq:base_opt_prob}. Third, we show that only minimizers of~\eqref{eq:base_opt_prob} can be asymptotically stable.
	
	\subsection*{Proof of Theorem~\ref{thm:main}: LaSalle function}
	
	In the following our choice of LaSalle function is inspired by ideas from singular perturbation analysis~\cite{khalil_nonlinear_2002, kokotovic_singular_1999}. Consider the function
	\begin{equation}\label{eq:lasalle_fct}
    	Z_\delta(x,u) = (1-\delta) V(u) + \delta W(u,x) \, ,
	\end{equation}
	where $0<\delta<1$ is a convex combination coefficient and
	\begin{align*}
		V(u)   &:= \Phi(H u + Rw, u) \\
		W(u,x) &:= (x - Hu - Rw)^T P (x - Hu - Rw) \, .
	\end{align*}
	
	Notice that $x - Hu - Rw$ implicitly defines the error of $x$ with respect to the steady state $Hu + Rw$.
	We show that $Z(x,u)$ is non-increasing along trajectories under the requirements of Theorem~\ref{thm:main}.

    \begin{lem} \label{lem:quadraticbound}
    Consider the function $Z_\delta$ as defined in \eqref{eq:lasalle_fct}. 
    Its Lie derivative along the trajectories of~\eqref{eq:ic_sys} satisfies
    \begin{equation} \label{eq:quadraticbound}
        \dot Z_\delta(u,x) \leq \begin{bmatrix} \|\psi\|  & \|\phi\| \end{bmatrix} \Lambda 
		 \begin{bmatrix} \| \psi\| \\ \|\phi\|\end{bmatrix} \, ,
    \end{equation}
    where 
    \begin{align*}
		\psi(x,u) & :=  \widetilde{H}^T \nabla \Phi(x,u)  \\
		\phi(x,u,w) & :=  x - Hu - Rw \, ,
    \end{align*}
    and
    \begin{equation} \label{eq:lasalle_matrix}
            	\Lambda := \begin{bmatrix}  
		 	 - \epsilon (1- \delta)  &
		 	  \frac{\epsilon}{2} \left( \ell (1-\delta) + \delta \beta \right)\\
		 	  \frac{\epsilon}{2} \left( \ell (1-\delta) + \delta \beta \right) &
		     - \frac{1}{2} \delta
		 \end{bmatrix}.  
    \end{equation}
    \end{lem}

    \begin{proof}
    The Lie derivative of $Z_\delta(x,u)$ along the trajectories of~\eqref{eq:ic_sys} is given by
	\begin{multline}\label{eq:lie_deriv}
		\dot{Z_\delta}(u,x) =  (1- \delta) \nabla_u V(u) g(x,u) \\ + \delta \nabla_x W(x,u) f(x,u) + \delta \nabla_u W(x,u) g(x,u) \, .
	\end{multline}
	Each of the terms in \eqref{eq:lie_deriv} can be bounded with the help of the two functions $\phi(x,u,w)$ and $\psi(x,u)$.
	Namely, for the first term we can write 	
	\begin{align*}
		&\nabla_u V(u) g(x,u) \\
		& \quad = - \epsilon \nabla \Phi(H u + Rw ,u)^T \widetilde{H} \widetilde{H}^T \nabla \Phi(x, u) \\
		& \quad = - \epsilon \left( \nabla \Phi(H u + Rw ,u) - \nabla \Phi(x,u) \right)^T \widetilde{H} \psi  \\	
		& \quad \qquad - \epsilon \nabla \Phi(x,u)^T \widetilde{H} \widetilde{H}^T \nabla \Phi(x,u) \\		
		& \quad = - \epsilon \left( \nabla \Phi(H u + Rw ,u) - \nabla \Phi(x,u) \right)^T \widetilde{H} \psi   - \epsilon \|\psi\|^2 \\
		& \quad \leq \epsilon \| \widetilde{H}^T \left( \nabla \Phi(H u + Rw ,u)  - \nabla \Phi(x,u) \right) \| \| \psi \|   - \epsilon \|\psi\|^2 \\
	    & \quad \leq \epsilon \ell \| x - H u - Rw \| \| \psi \|   - \epsilon \| \psi\|^2 \\			
		& \quad \leq \epsilon \ell \| \phi \| \| \psi \| - \epsilon \|\psi\|^2 \, ,
	\end{align*}
	where we have used the definition of $\ell$ from Assumption~\ref{eq:ass_lip}.

	For the second term in~\eqref{eq:lie_deriv} we get
	\begin{align*}
	   & \nabla_x W(x,u) f(x,u) = \\
		& \quad = (x - Hu - Rw)^T P (A x + B u + Qw) \\
		& \quad = (x - Hu - Rw)^T P A ( x - H u - Rw) \\
		& \quad = \tfrac{1}{2} \phi^T (PA + A^T P) \phi  \leq - \tfrac{1}{2} \| \phi\|^2 \, ,
	\end{align*}
	and finally for the third term we get
	\begin{align*}
		& \nabla_u W(x,u) g(x,u) =\\
		& \quad = - \epsilon (x - Hu - Rw)^T P H \widetilde{H}^T \nabla \Phi(x,u)
		 \leq   \epsilon \beta  \| \phi\| \| \psi \| \, .
	\end{align*}
	Therefore the Lie derivative of $Z_\delta$ is bounded by a quadratic form in $\|\phi\|$ and $\|\psi\|$, that can be rewritten in matricial form as in \eqref{eq:quadraticbound}.
\end{proof}

We can characterize under which conditions $ \dot Z_\delta(u,x) \leq 0$ by verifying when $\Lambda$ is negative definite, as proposed in the following lemma.

\begin{lem} \label{lem:main_bound}
	Consider $\Lambda$ from~\eqref{eq:lasalle_matrix}.
	Whenever $\epsilon$ satisfies the bound \eqref{eq:main_bound}, then $\Lambda \prec 0 $ for some $\delta^\star \in (0, 1)$.
\end{lem}

\begin{proof}
	The proof is standard~\cite{kokotovic_singular_1999}. Namely, $\Lambda$ has negative diagonal elements and hence as a $2\times 2$-matrix is negative definite if and only if
	\begin{align*}
		0 < \frac{\epsilon}{2} \delta (1-\delta) - \frac{\epsilon^2}{4} \left( l(1-\delta) + \beta \delta \right)^2 \, .
	\end{align*}
	Equivalently, one can solve for $\epsilon$ and get
	\begin{align*}
		 \epsilon < \frac{2 \delta (1- \delta)}{\left( l(1-\delta) + \beta \delta \right)^2} \, .
	\end{align*}
	The right-hand side of this expression can be shown to have a maximum at $\delta^\star = \frac{ \ell}{\ell + \beta}$ which is $\epsilon^\star = \frac{1}{2 \beta \ell}$.
\end{proof}

    \subsection*{Proof of Theorem~\ref{thm:main}: Convergence to Critical Points}
    
    \begin{lem}
    Let $\epsilon$ satisfy \eqref{eq:main_bound}, and let $\delta^\star$ be defined as in Lemma~\ref{lem:main_bound}.
	Then $\dot Z_{\delta^\star}(x,u) \le 0$ for all $(x,u)$, and $\dot Z_{\delta^\star}(x,u) = 0$ if and only if $(x,u)\in E$, where
	\begin{equation}
	    E = \left\{
	        (x,u) \,|\, 
	        x = Hu + Rw,
	        \nabla \Phi(x^\star, u^\star) \in \ker \widetilde{H}^T 
		    \right\}.
    \label{eq:setE}
	\end{equation}
	Moreover, every point in $E$ is an equilibrium.
	\label{lem:vdotzero}
	\end{lem}
	
    \begin{proof}
    We know from Lemma~\ref{lem:quadraticbound} and Lemma~\ref{lem:main_bound} that $\dot{Z}_{\delta^\star}(x,u)$ is upper-bounded by a non-positive function of $(x,u)$, and therefore $\dot Z_{\delta^\star}(x,u) \le 0$.
    We also know from the same lemmas that the set of points $(x,u)$ where $\dot{Z}_{\delta^\star}(x,u) = 0$ is a subset of the set of points where $\|\psi\|=0$ and $\|\phi\|=0$, which corresponds to the set $E$ defined in \eqref{eq:setE}.
	It can be verified by inspection of the system dynamics \eqref{eq:ic_sys} that all the elements of $E$ are equilibria, hence $E$ is the exhaustive set of solutions to $\dot{Z}_{\delta^\star}(x,u) = 0$.
    \end{proof}

    \begin{lem}
        Let $\epsilon$ satisfy \eqref{eq:main_bound}, and let $\delta^\star$ be defined as in Lemma~\ref{lem:main_bound}.
	    The sublevel sets of $Z_{\delta^\star}(x,u)$ are compact and positively invariant for the dynamics \eqref{eq:ic_sys}.
	\label{lem:positiveinvariance}
	\end{lem}
	
	\begin{proof}
	Consider a sublevel set $\Omega_c:=\{(x,u) \,|\, Z_{\delta^\star}(x,u) \le c\}$, for some $c\in \bbR$.

    First notice that, as $W(x,u)\ge0$, $Z_{\delta^\star}(x,u) \le c$ implies $V(u) \le c$.
    But $V(u)=\tilde{\Phi}^w(u)$, which by Assumption~\ref{ass:cvx} has compact sublevel sets, and therefore there exist $U$ such that $\|u\| \le U$ for all $(x,u)\in \Omega_c$.

    On the compact set $\{u\,|\,\|u\| \le U\}$ the continuous function $V(u)$ is also lower bounded by some value $L$.
    We therefore have that $W(x,u) \le c - L$ in $\Omega_c$.
    As $P$ is positive definite, we must have that 
    $\|x-Hu-Rw\|^2 \le (c-L)/\lambda_{\min}(P)$.
    We then have
    \begin{align*}
        \|x\| &\le \|x-Hu-Rw\| + \|Hu\| + \|Rw\|\\
            &\le \sqrt{\frac{c - L}{\lambda_{\min}}} + \|H\| \|u\| + \|Rw\| \\
            &\le \sqrt{\frac{c - L}{\lambda_{\min}}} + \|H\| U + \|Rw\|,
    \end{align*}
    and therefore $\|x\|$ is also bounded for all $(x,u) \in \Omega_c$.
    
    Positive invariance of $\Omega_c$ follows from the fact that $\dot Z_{\delta^\star}(x,u) \le 0$ (via Lemma~\ref{lem:vdotzero}).
	\end{proof}

    Lemmas \ref{lem:vdotzero} and \ref{lem:positiveinvariance} allow us to use LaSalle's invariance principle \cite[p.128]{khalil_nonlinear_2002} (which we briefly recall hereafter).
    
	\begin{thm*}[LaSalle's Invariance Principle]
		Let $\Omega \subset D (\subset \bbR^n)$ be a compact set that is positively invariant with respect to an autonomous system $\dot \eta = F(\eta)$, where $F$ is locally Lipschitz.
		Let $Z : D \rightarrow \bbR$ be a continuously differentiable function such that $\dot{Z}(\eta) \leq 0$ in $\Omega$.
		Let $E$ be the set of all the points in $\Omega$ where $\dot{Z}(\eta) = 0$.
		Let $M$ be the largest invariant set in $E$.
		Then every solution starting in $\Omega$ approaches $M$ as $t \rightarrow \infty$.
	\end{thm*}
	
   We therefore conclude that every solution will approach the largest invariant set $M$ contained in $E$, as defined in \eqref{eq:setE}, which is $E$ itself.
   As shown with \eqref{eq:firstorderoptimality}, each point in $E$ is a critical point for~\eqref{eq:base_opt_prob}, i.e., it satisfies the first-order optimality conditions for the optimization problem at hand.

	\subsection*{Proof of Theorem~\ref{thm:main}: Asymptotic Stability of Strict Minimizers}
	
    To prove the final statement of Theorem~\ref{thm:main}, we show that an asymptotically stable equilibrium $z^\star := (x^\star, u^\star)$ of~\eqref{eq:ic_sys} is a strict optimizer of~\eqref{eq:base_opt_prob}. In other words, on the feasible set $\mathcal{Z} := \{ (x,u) \, | \, x = Hu + Rw \}$ the point $z^\star$ is a strict minimizer of $\Phi(u,x)$.
    
    We first show that $z^\star$ is a local minimizer. Consider a neighborhood $\mathcal{V} \subset \mathcal{Z}$ of $z^\star$ such that the respective trajectories of~\eqref{eq:ic_sys} and of the reduced system~\eqref{eq:grad_desc} both initialized at any $z_0$ in $\mathcal{V}$ converge to $z^\star$. (Such a neighborhood can be constructed as the intersection of the regions of attraction of $z^\star$ under~\eqref{eq:ic_sys} and under~\eqref{eq:grad_desc} since $x^\star$ is asymptotically stable for both of them.) Even though $z_0, z^\star \in \mathcal{Z}$, the trajectory $z(t)$ of~\eqref{eq:ic_sys} starting at $z_0$ and converging to $z^\star$ may not be contained in $\mathcal{Z}$. On the other hand, the trajectory $z'(t)$ of~\eqref{eq:grad_desc} starting at $z_0$ and converging to $z^\star$ will always lie in $\mathcal{Z}$ by definition. Furthermore, $\Phi$ is non-increasing along $z'(t)$ and therefore $\Phi(z^\star) \leq \Phi(z_0)$. Since $z_0$ is arbitrary in $\mathcal{V}$ we therefore conclude that $z^\star$ is a minimizer of $\Phi$ on $\mathcal{V}$.

    For strict minimality of $z^\star$ on $\mathcal{Z}$, let $\bar{z} \in \mathcal{V}$, but with $\Phi(\bar{z}) \leq \Phi(z^\star)$.
    Every solution $z(t)$ of~\eqref{eq:grad_desc} with $z(0) = \bar{z}$ nevertheless converges to $z^\star$ by assumption on $\mathcal{V}$, and hence it follows that $\nabla \tilde{\Phi}^w(u(t)) = 0$ for almost all $t \geq 0$ (i.e., $z(t)$ is an equilibrium for almost all $t$). This, however, contradicts the fact that $z^\star$ is asymptotically stable. 
    Therefore, we conclude that asymptotically stable equilibria are strict minimizers.
    
    This completes the proof of Theorem~\ref{thm:main}.

\section{Application to Power Systems}	
\label{sec:application}

Our forthcoming power system model is deliberately simple and idealized to illustrate the application of the stability bound derived in the previous section. However, the full generality of the proposed design method allows to do the same for much more complicated and high-dimensional linear models and even gain insights for nonlinear models on a local scale.

We consider a connected electric power transmission system with $r$ buses and $s$ lines. To every node $i$ we associate a phase angle $\theta_i$ which we collect in the vector $\theta \in \bbR^r$. We assume that lines are lossless, voltages are close to nominal and phase angle differences are small. Hence, we approximate power flow by the classical \emph{DC power flow} equations
\begin{align*}
    p^E = \laplac \theta  \, ,
\end{align*}
where $p^E \in \bbR^r$ denotes the power injection into the grid at each node. Further, $\laplac \in \bbR^{r \times r}$ is the graph Laplacian of the network defined as
\begin{align*}
    \left[ \laplac \right]_{ij} :=
    \begin{cases}
        \sum_{k \in \mathcal{N}(i)} b_{ik} \quad & i = j \\
        - b_{ij} \quad & j \in \mathcal{N}(i) \\
        0 \quad & \text{otherwise, } 
    \end{cases}
\end{align*}
where $b_{ij} > 0$ denotes the susceptance of the line connecting nodes $i$ and $j$, and $\mathcal{N}(i)$ denotes the set of all nodes that connect to node $i$.

Further, given the phase angles $\theta$, the line flows are computed as $p^\ell_{ij} = b_{ij}( \theta_i - \theta_j)$ which we write in matrix notation as
\begin{align*}
  p^\ell = \laplac^\ell \theta \, ,
 \end{align*}
where $p^\ell \in \bbR^s$ and $\laplac^\ell \in \bbR^{s \times r}$.

We assume that a single generator is connected to each node.
Every node is hence characterized by a deviation $\omega_i$ from the nominal frequency.
Off-nominal frequencies cause phase angles to drift according to 
\begin{align}\label{eq:angle_dyn}
   \dot{\theta} = \omega \, .
\end{align}
Further, we assume that $\omega$ follows the swing dynamics
\begin{align}\label{eq:swing}
    \dot \omega = - M^{-1} \left(D \omega + \laplac \theta - p^M + p^L(t)\right) \, ,
\end{align}
where $M \in \bbR^{r \times r}$ and $D \in \bbR^{r \times r}$ are full rank diagonal matrices that collect (positive) \emph{inertia} and \emph{damping} constants, respectively. Further, $p^L(t) \in \bbR^r$ denotes a (possibly time-varying) vector of exogenous power demands at each bus.

Finally, the mechanical power output $p^M \in \bbR^r$ of the generators is subject to simple turbine-governor dynamics
\begin{align}\label{eq:gov}
    \dot p^M = - T^{-1} \left(R^{-1} \omega +  p^M - p^C(t) \right) \, ,
\end{align}
where $T, R \in \bbR^{r \times r}$ are full-rank, diagonal matrices that describe the (positive) \emph{control gain} and \emph{droop constants} respectively, and $p^C(t) \in \bbR^r$ denotes generation setpoints.

Together, the ODEs~\eqref{eq:angle_dyn},~\eqref{eq:swing}, and~\eqref{eq:gov} approximate the power system dynamics of a network with synchronous generators that are equipped with frequency droop controllers (i.e., primary frequency control). 

\begin{remark}
It is possible to consider nodes without generators where instead of~\eqref{eq:swing} the algebraic constraint $p^E_i + p^L_i = 0$ holds and thus leading to a \emph{differential-algebraic system}. Our modeling assumption can be recovered by applying a Kron reduction that eliminates all non-generator buses~\cite{dorfler_kron_2013}.
\end{remark}

\subsubsection{Input-Output Representation}

For our purposes, we write~\eqref{eq:angle_dyn},~\eqref{eq:swing} and~\eqref{eq:gov} in the form of a linear time-invariant system as given by~\eqref{eq:sys_dyn} with $p^C(t)$ as input and $p^L(t)$ as a disturbance. 
Namely, we have 
\begin{align*}
    x := \begin{bmatrix}
            \theta \\ \omega \\ p^M
        \end{bmatrix} \in \bbR^{3r}\,,  \, u := p^C(t) \, \text{, and }
        w :=  p^L(t) \, 
\end{align*}
with the system matrices defined as
\begin{align*}
    A := \begin{bmatrix}
        \bbzero_{r\times r} & \bbI_r & \bbzero_{r\times r}  \\
        - M^{-1} \laplac & - M^{-1} D & M^{-1} \\
        \bbzero_{r\times r}  & -T^{-1}R^{-1} & -T^{-1}
    \end{bmatrix} \, ,\\\\
    B := \begin{bmatrix}
    	\bbzero_{r\times r} \\ \bbzero_{r\times r} \\ T^{-1}
    \end{bmatrix} 
    \quad \text{and} \quad
    Q := \begin{bmatrix}
        \bbzero_{r\times r} \\ -M^{-1} \\ \bbzero_{r\times r}
        \end{bmatrix} \, .
\end{align*}

We also assume that the observable outputs of the system consist of the frequency at node 1 and the line flows, i.e., we define the output as $y := Cx$ where  
\begin{align*}
    C &:= \begin{bmatrix}
    \bbzero_{1 \times r} & \begin{matrix} 1 & \bbzero_{1 \times 2r-1} \end{matrix} \\
    \laplac^\ell & \bbzero_{s \times 2r}
    \end{bmatrix} \, .
\end{align*}

\subsubsection{Steady-State Map} By definition, $A$ cannot be asymptotically stable (and therefore not directly satisfy Assumption~\ref{ass:stab}) because the Laplacian $\laplac$ has a non-empty kernel that is associated with the translation invariance of the phase angles $\theta$. The output $y$ is however not affected by this unstable mode. Nevertheless, to apply Theorem~\ref{thm:main}, we need to eliminate the unstable mode of $A$ using an appropriate coordinate transformation of the internal state $x$. As a consequence Assumption~\ref{ass:stab} can be tested by solving the Lyapunov equation~\eqref{eq:ass_lyap} for this reduced system.

\subsection{Optimization Problem}
We consider the problem of optimizing an economic cost while at the same time controlling frequency and line congestion. Namely, we consider the problem
\begin{subequations}\label{eq:opt_prob}
\begin{align}
  \underset{u,x}{\text{minimize}} &\quad f(u) \\
  \text{subject to} & \quad x = H u + G w \\
            & \quad \underline{u} \leq  u \leq \overline{u}  \label{eq:opcon1}\\
            & \quad \underline{y} \leq  C x \leq \overline{y} \, , \label{eq:opcon2}
\end{align}
\end{subequations}
where $f: \bbR^r \rightarrow \bbR$ is a convex cost function and $\underline{u}, \overline{u} \in \bbR^r$ denote lower and upper limits on power generation setpoints\footnotemark.
Further, $\underline{y}, \overline{y} \in \bbR^{s+1}$ denote constraints on the output given as
$ \underline{y} := \left[ \begin{smallmatrix} 0 \\ -\overline{p}^\ell \end{smallmatrix} \right]$ and $ \overline{y} := \left[ \begin{smallmatrix} 0 \\ \overline{p}^\ell \end{smallmatrix} \right]$ where $\overline{p}^\ell \in \bbR^s$ denotes the line ratings.

\footnotetext{The economic cost and input constraints apply to set-points rather than actual power generation because, from the viewpoint of a grid operator, the (re)dispatch of generation is subject to contractual agreements and reserve limits rather than actual generation cost and plant constraints.}

In order to comply with our theoretical results we treat the operational constraints \eqref{eq:opcon1}-\eqref{eq:opcon2} as soft constraints by considering the augmented objective function
\begin{align*} 
    \Phi(x,u) := f(u) + \rho\left(\begin{bmatrix} u \\ Cx \end{bmatrix} - \begin{bmatrix} \overline{u} \\ \overline{y} \end{bmatrix}\right) + \rho\left(\begin{bmatrix} \underline{u} \\ \underline{y} \end{bmatrix} - \begin{bmatrix} u \\ Cx \end{bmatrix} \right) \, ,
\end{align*}
where $\rho(w) := \tfrac{1}{2} \| \max \{ 0, w \} \|^2_\Xi = \tfrac{1}{2}\sum_i \xi_i \|\max \{ 0, w_i \} \|^2$ is a unilateral quadratic penalty function with $\Xi = \text{diag}(\xi)$ being a weighting matrix. 
Note that $\rho$ is continuously differentiable and hence $\nabla \Phi$ can be evaluated and used as feedback control law of the form~\eqref{eq:fb_law}.

\begin{remark}
The optimization problem \eqref{eq:opt_prob} can be specialized in order to model individual ancillary services separately. For example, by setting $f(u)\equiv 0$, no constraints \eqref{eq:opcon1} on $u$, and output constraints \eqref{eq:opcon1} as $0 \le \omega_1 \le 0$, one obtains the augmented objective function
$\Phi(x,u) = \tfrac{1}{2} \omega_1^2$.
The gradient descent flow $-\epsilon \widetilde{H} \nabla \Phi(x,u)$ corresponds to
\[
\dot p^C = -\frac{\epsilon}{\bbone^T (D + R^{-1})\bbone} \bbone \omega_1
\]
where $\bbone$ is the vector of all ones.
As discussed in \cite{grammatico2017}, this corresponds to the standard industrial solution for AGC (with homogeneous participation factors).
\end{remark}

\section{Simulation Results}

We simulate the the proposed feedback interconnection using the standard IEEE 118-bus power system test case~\cite{zimmerman_matpower_2011} that we augment with reasonable, randomized values for $M$, $D$, $T$ and $R$ with mean \SI{5}{\s}, \SI{3}, \SI{4}{\s}, and \SI{0.25}{\Hz/\pu}, respectively. Furthermore, we impose a line flow limit of \SI{2.5}{\pu} on all lines which results in several lines running the risk of congestion under normal operation.
For the control algorithm the penalty parameters have been chosen as $\xi = 10^3$ for all constraints on relating power generation and line flows, and $\xi = 10^7$ for the frequency deviation $\omega_1$. This choice is mainly motivated by the difference in the order of magnitude of the respective quantities. To numerically integrate~\eqref{eq:ic_sys} we use an exact discretization of the LTI system~\eqref{eq:sys_dyn} in combination with a simple explicit Euler scheme for the interconnected system~\ref{eq:ic_sys}.

Our primary interest lies in the tightness of the stability bound provided by Theorem~\ref{thm:main}, namely for the setup under consideration we get $\epsilon^\star = 1.069\times 10^{-6}$. Experimentally we find that for $\epsilon$ smaller than approximately  $5\epsilon^\star$ the feedback interconnection is nevertheless asymptotically stable whereas for larger $\epsilon$ we observe instability.

This conservatism of our bound is to be expected, especially considering the limited amount of information and computation required to evaluate~$\epsilon^\star$. 
Nevertheless, we consider $\epsilon^\star$ to be of practical relevance, in particular since the bound comes with the guarantee that for every $\epsilon < \epsilon^\star$ the interconnected system is asymptotically stable. 

Previous works on real-time feedback-based optimization schemes~\cite{hauswirth_online_2017, dallanese2018, gan2016} have considered only the interconnection with a steady-state map and have shown that feedback controllers can reliably track the solution of a time-varying OPF problem, even for nonlinear setups. Hence, a second insight of our simulations is that the interconnection of a gradient-based controller with a dynamical system instead of an algebraic steady-state map does not significantly deteriorate the long-term tracking performance.

Hence, in Figure~\ref{fig:plot} we show the simulation results over a time span of \SI{300}{s} in which the feedback controller tries to track the solution of a standard DC OPF problem\footnote{Any solution of~\eqref{eq:opt_prob} has zero frequency deviation from the nominal value, hence the variable $\omega_1$ can be eliminated from the problem formulation, which results in a standard DC OPF problem.} of the form~\eqref{eq:opt_prob} under time-varying loads, i.e., non-constant disturbance $w(t)$. The optimal cost of this \emph{instantaneous} DC OPF problem is illustrated in the first panel of Figure~\ref{fig:plot} by the the dashed line. The minor violation of line and frequency constraints observed in the simulation is a consequence of control design that is based on soft constraints.

Additionally, we simulate the effect of a generator outage at \SI{100}{s} and a double line tripping at \SI{200}{s}. Both of which do not jeopardize overall stability. Furthermore, we do not update the steady-state map $H$ after the grid topology change caused by the line outages. Nevertheless, the controller with the inexact model achieves very good tracking performance as illustrated in the first panel of Figure~\ref{fig:plot}.

Overall, this simulation shows the robustness of feedback-based optimization against i) underlying dynamics, ii) disturbances in the form of load changes, and iii) model inaccuracy in the form of topology mismatch.

\begin{figure}[tb]
\def\scalingfactor{0.67}
\raggedright
\vspace{.1cm}
\includegraphics[scale=\scalingfactor,right]{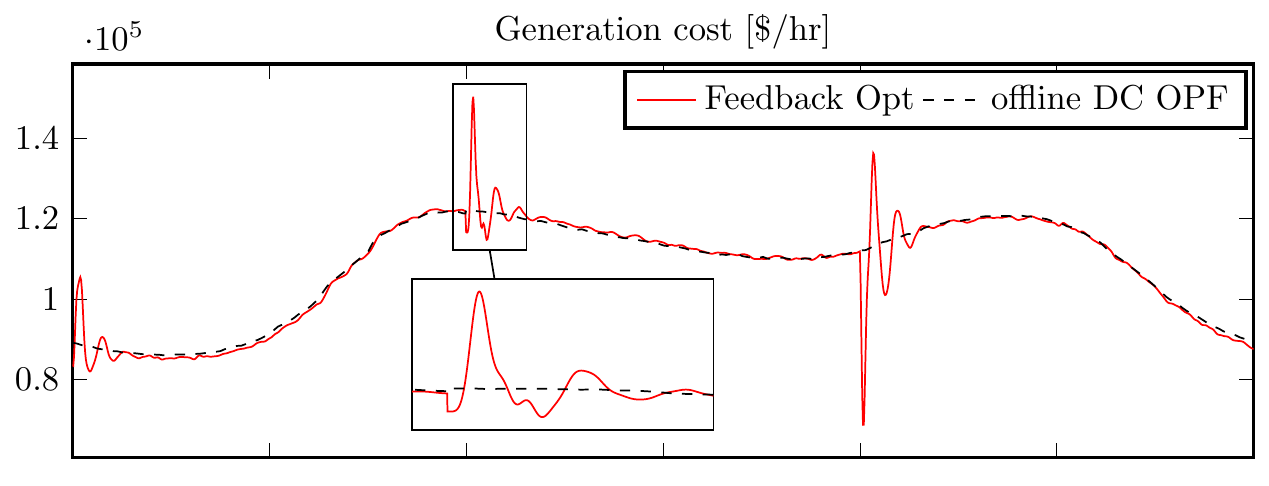}
\includegraphics[scale=\scalingfactor,right]{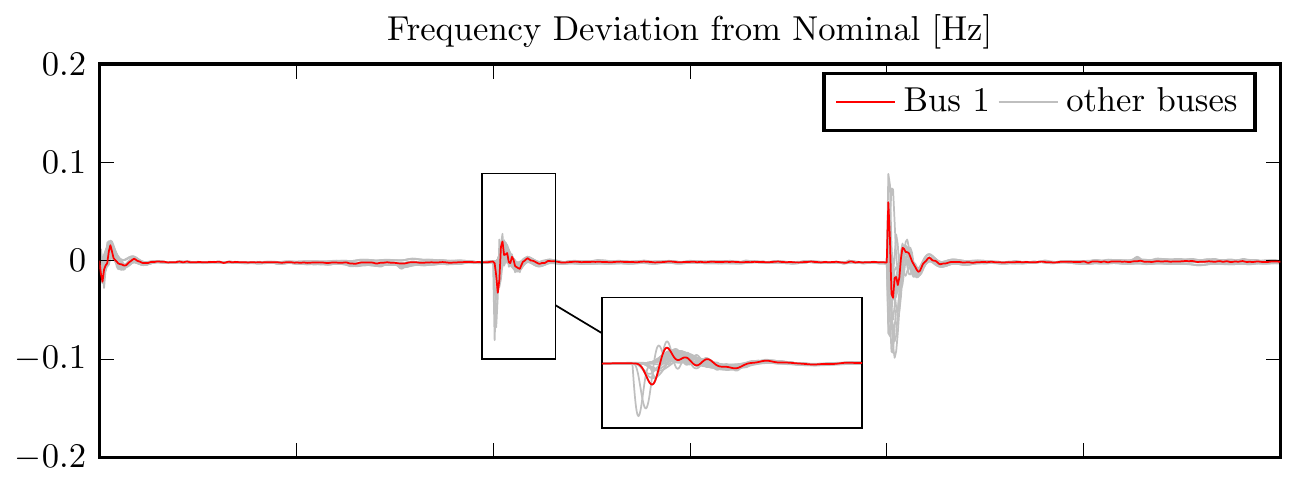}
\includegraphics[scale=\scalingfactor,right]{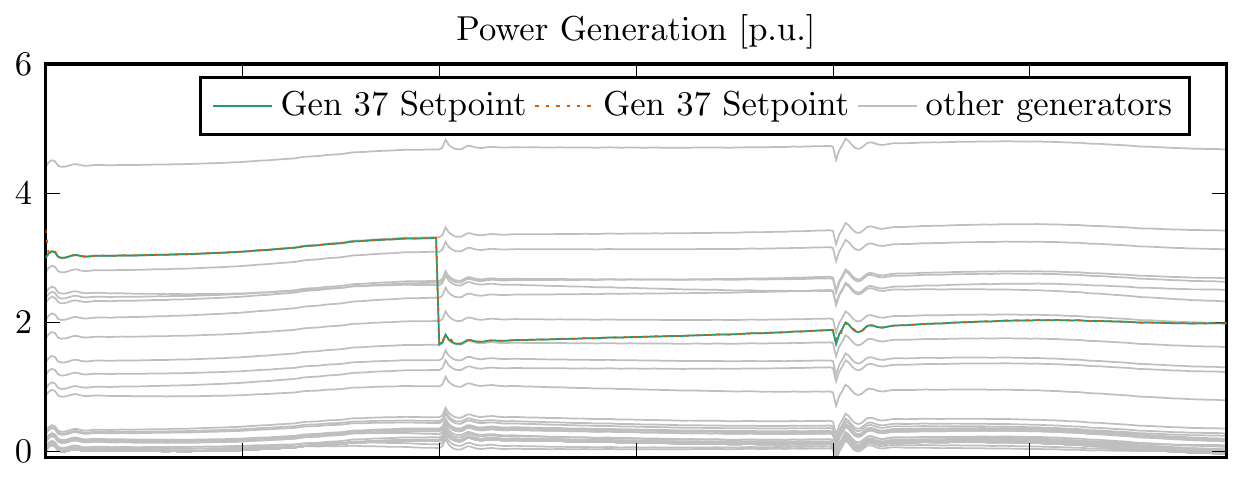}
\includegraphics[trim={-2.26mm 0 2.26mm 0}, scale=\scalingfactor,right]{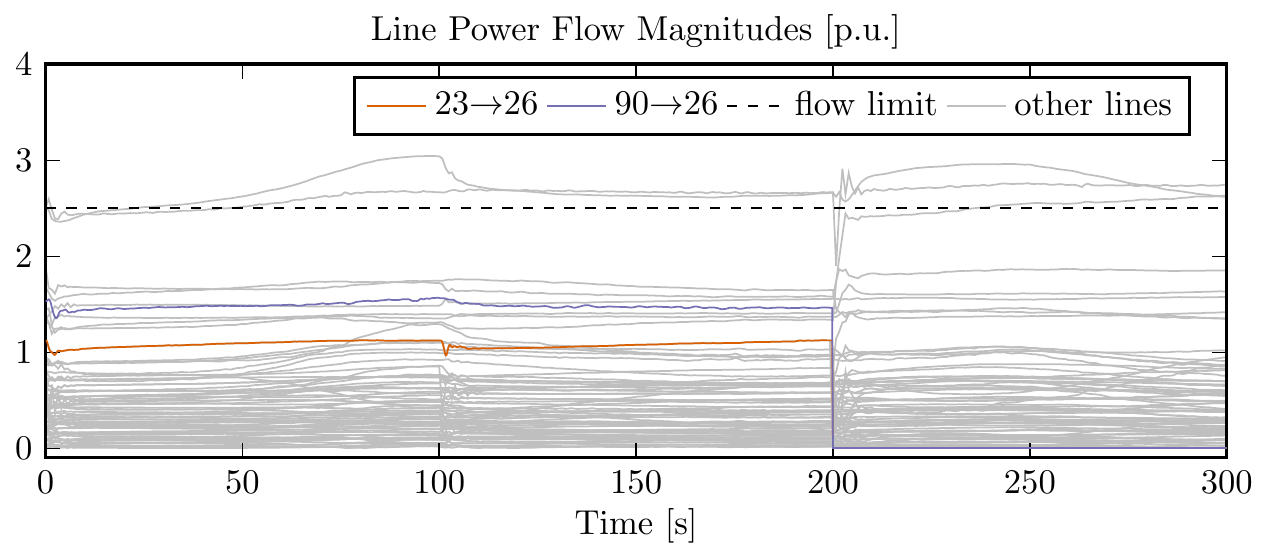}

\caption{Simulation results for the IEEE 118-bus test case. An outage of a \SI{200}{\MW} generation unit happens at \SI{100}{\s} (producing approx. \SI{175}{\MW} at the time of the outage) and results in the loss of half the generation capacity at the corresponding bus (the bus power injection is highlighted in the third panel). A double line tripping happens at \SI{200}{\s} (the corresponding line power flows are highlighted in the fourth panel).}
\label{fig:plot}
\end{figure}

\section{Conclusions}

In this paper we have considered the problem of designing  feedback controllers for an LTI plant that ensure i) internal stability of the interconnected system and ii) convergence of the system trajectories to the set of critical points of a prescribed optimization program.

The proposed solution consists in the design of a gradient-based feedback law which is based on a time-scale separation argument, i.e., replacing the plant with its algebraic steady-state map.
We then have derived a prescription on how to tune the gain of the feedback control loop in order to guarantee stable operation of the interconnected system.
The resulting bound is a direct function of meaningful quantities of both the plant and of the cost function, and requires very weak assumptions.
Moreover, the methodology that allows to derive this bound has very broad applicability.
We expect to extend the analysis to projected gradient and primal-dual schemes in future work.

We have applied this design approach to frequency regulation and economic re-dispatch mechanisms on an electric power grid. For the test case under consideration we have shown that our stability bound is only mildly conservative and allows to tune the feedback control law in an effective manner: The closed-loop system is capable of tracking the time-varying optimizer promptly, even in the presence of significant disturbances. 
The proposed design method can be readily applied to more complex and high dimensional power system models, enabling the use of feedback optimization for multiple ancillary services.

\bibliographystyle{IEEEtran}
\bibliography{IEEEabrv,bibliography_static}

\begin{thebibliography}{10}
\providecommand{\url}[1]{#1}
\csname url@samestyle\endcsname
\providecommand{\newblock}{\relax}
\providecommand{\bibinfo}[2]{#2}
\providecommand{\BIBentrySTDinterwordspacing}{\spaceskip=0pt\relax}
\providecommand{\BIBentryALTinterwordstretchfactor}{4}
\providecommand{\BIBentryALTinterwordspacing}{\spaceskip=\fontdimen2\font plus
\BIBentryALTinterwordstretchfactor\fontdimen3\font minus
  \fontdimen4\font\relax}
\providecommand{\BIBforeignlanguage}[2]{{%
\expandafter\ifx\csname l@#1\endcsname\relax
\typeout{** WARNING: IEEEtran.bst: No hyphenation pattern has been}%
\typeout{** loaded for the language `#1'. Using the pattern for}%
\typeout{** the default language instead.}%
\else
\language=\csname l@#1\endcsname
\fi
#2}}
\providecommand{\BIBdecl}{\relax}
\BIBdecl

\bibitem{bonvin2013}
G.~Francois and D.~Bonvin, ``Measurement-based real-time optimization of
  chemical processes,'' \emph{Advances in Chemical Engineering}, vol.~43, pp.
  1--508, 2013.

\bibitem{kelly1998}
F.~P. Kelly, A.~K. Maulloo, and D.~K.~H. Tan, ``Rate control for communication
  networks: shadow prices, proportional fairness and stability,'' \emph{Journal
  of the Operational Research Society}, vol.~49, no.~3, pp. 237--252, Mar.
  1998.

\bibitem{dallanese2016}
E.~Dall'Anese, S.~V. Dhople, and G.~B. Giannakis, ``Photovoltaic inverter
  controllers seeking ac optimal power flow solutions,'' \emph{IEEE
  Transactions on Power Systems}, vol.~31, no.~4, pp. 2809--2823, Jul. 2016.

\bibitem{gan2016}
L.~Gan and S.~H. Low, ``An online gradient algorithm for optimal power flow on
  radial networks,'' \emph{IEEE Journal on Selected Areas in Communications},
  vol.~34, no.~3, pp. 625--638, Mar. 2016.

\bibitem{hauswirth_projected_2016}
A.~Hauswirth, S.~Bolognani, G.~Hug, and F.~D\"orfler, ``Projected gradient
  descent on {{Riemannian}} manifolds with applications to online power system
  optimization,'' in \emph{54th {{Annual Allerton Conference}} on
  {{Communication}}, {{Control}}, and {{Computing}}}, Sep. 2016, pp. 225--232.

\bibitem{hauswirth_online_2017}
A.~Hauswirth, A.~Zanardi, S.~Bolognani, F.~D\"orfler, and G.~Hug, ``Online
  optimization in closed loop on the power flow manifold,'' in \emph{2017
  {{IEEE Manchester PowerTech}}}, Jun. 2017, pp. 1--6.

\bibitem{dallanese2018}
E.~Dall’Anese and A.~Simonetto, ``Optimal power flow pursuit,'' \emph{IEEE
  Transactions on Smart Grid}, vol.~9, no.~2, pp. 942--952, Mar. 2018.

\bibitem{liu2018}
J.~Liu, J.~Marecek, A.~Simonetta, and M.~Taka{\v{c}}, ``A coordinate-descent
  algorithm for tracking solutions in time-varying optimal power flows,'' in
  \emph{Power Systems Computation Conference (PSCC)}, Jun. 2018.

\bibitem{low2014}
C.~Zhao, U.~Topcu, N.~Li, and S.~Low, ``Design and stability of load-side
  primary frequency control in power systems,'' \emph{IEEE Transactions on
  Automatic Control}, vol.~59, no.~5, pp. 1177--1189, May 2014.

\bibitem{cady2015}
S.~T. Cady, A.~D. Dom{\'i}nguez-Garc{\'i}a, and C.~N. Hadjicostis, ``A
  distributed generation control architecture for islanded ac microgrids,''
  \emph{IEEE Transactions on Control Systems Technology}, vol.~23, no.~5, pp.
  1717--1735, Sep. 2015.

\bibitem{mallada2017}
E.~Mallada, C.~Zhao, and S.~Low, ``Optimal load-side control for frequency
  regulation in smart grids,'' \emph{IEEE Transactions on Automatic Control},
  vol.~62, no.~12, pp. 6294--6309, Dec. 2017.

\bibitem{bolognani2015}
S.~Bolognani, R.~Carli, G.~Cavraro, and S.~Zampieri, ``Distributed reactive
  power feedback control for voltage regulation and loss minimization,''
  \emph{IEEE Transactions on Automatic Control}, vol.~60, no.~4, pp. 966--981,
  Apr. 2015.

\bibitem{todescato2018}
M.~Todescato, J.~W. Simpson-Porco, F.~Dörfler, R.~Carli, and F.~Bullo,
  ``Online distributed voltage stress minimization by optimal feedback reactive
  power control,'' \emph{IEEE Transactions on Control of Network Systems},
  vol.~5, no.~3, pp. 1467--1478, Sep. 2018.

\bibitem{survey2017}
D.~K. Molzahn, F.~D{\"o}rfler, H.~Sandberg, S.~H. Low, S.~Chakrabarti,
  R.~Baldick, and J.~Lavaei, ``A survey of distributed optimization and control
  algorithms for electric power systems,'' \emph{IEEE Transactions on Smart
  Grid}, vol.~8, no.~6, pp. 2941--2962, Nov. 2017.

\bibitem{nelson_integral_2017}
Z.~E. Nelson and E.~Mallada, ``An integral quadratic constraint framework for
  real-time steady-state optimization of linear time-invariant systems,''
  \emph{ArXiv171010204 Cs Math}, Oct. 2017.

\bibitem{colombino_online_2018}
M.~Colombino, E.~Dall'Anese, and A.~Bernstein, ``Online optimization as a
  feedback controller: {{Stability}} and {{Tracking}},'' \emph{ArXiv180509877
  Math}, May 2018.

\bibitem{li2018}
X.~Zhang, A.~Papachristodoulou, and N.~Li, ``Distributed control for reaching
  optimal steady state in network systems: An optimization approach,''
  \emph{IEEE Transactions on Automatic Control}, vol.~63, no.~3, pp. 864--871,
  Mar. 2018.

\bibitem{han_computational_2018}
S.~{Han}, ``{Computational Convergence Analysis of Distributed Gradient Descent
  for Smooth Convex Objective Functions},'' \emph{ArXiv1810.00257 Math}, Sep.
  2018.

\bibitem{luenberger_linear_1984}
D.~G. Luenberger and Y.~Ye, \emph{Linear and Nonlinear Programming},
  4th~ed.\hskip 1em plus 0.5em minus 0.4em\relax Cham, Switzerland: {Springer},
  1984.

\bibitem{khalil_nonlinear_2002}
H.~K. Khalil, \emph{Nonlinear {{Systems}}}, 3rd~ed.\hskip 1em plus 0.5em minus
  0.4em\relax Upper Saddle River, NJ: {Prentice Hall}, 2002.

\bibitem{kokotovic_singular_1999}
P.~Kokotovic, H.~Khalil, and J.~O'Reilly, \emph{Singular {{Perturbation
  Methods}} in {{Control}}: {{Analysis}} and {{Design}}}, ser. Classics in
  Applied Mathematics.\hskip 1em plus 0.5em minus 0.4em\relax Philadelphia, PA:
  {Society for Industrial and Applied Mathematics}, 1999, no.~25.

\bibitem{dorfler_kron_2013}
F.~D\"orfler and F.~Bullo, ``Kron {{Reduction}} of {{Graphs With Applications}}
  to {{Electrical Networks}},'' \emph{IEEE Trans. Circuits Syst. Regul. Pap.},
  vol.~60, no.~1, pp. 150--163, Jan. 2013.

\bibitem{grammatico2017}
F.~D{\"o}rfler and S.~Grammatico, ``Gather-and-broadcast frequency control in
  power systems,'' \emph{Automatica}, vol.~79, pp. 296 -- 305, 2017.

\bibitem{zimmerman_matpower_2011}
R.~D. Zimmerman, C.~E. Murillo-Sanchez, and R.~J. Thomas, ``Matpower:
  Steady-state operations, planning, and analysis tools for power systems
  research and education,'' \emph{IEEE Transactions on Power Systems}, vol.~26,
  no.~1, pp. 12--19, Feb 2011.

\end{thebibliography}
\end{document}